\documentclass{cedram-aif}

\usepackage{amsfonts, amssymb, amscd}
\usepackage{bm}
\usepackage{verbatim}
\usepackage{amssymb}
\usepackage{mathrsfs}
\usepackage{graphicx}
\usepackage[all]{xy}
\usepackage{tikz,tikz-cd}
\usepackage{subcaption}
\usepackage{subfiles}
\usepackage[toc,page]{appendix}
\usepackage{mathtools}
\usepackage{comment}
\usepackage{enumerate}
\usepackage{enumitem}
\usepackage[normalem]{ulem}
\usepackage{graphicx}
\usepackage{appendix}
\usepackage{hyperref}

\newcommand{\CC}{\mathbb{C}}
\newcommand{\FF}{\mathbb{F}}

\newcommand{\Qq}{\mathbb{Q}}
\newcommand{\QQ}{\mathbb{Q}}
\newcommand{\Rr}{\mathbb{R}}
\newcommand{\RR}{\mathbb{R}}

\newcommand{\ZZ}{\mathbb{Z}}

\newcommand{\Exc}{\operatorname{Exc}}

\newcommand{\Supp}{\operatorname{Supp}}

\newcommand{\bfD}{\mathbf{D}}
\newcommand{\bfDiv}{\mathbf{Div}}

\newcommand{\bfM}{\mathbf{M}}
\newcommand{\bfN}{\mathbf{N}}

\newcommand{\Bb}{\mathcal{B}}

\newcommand{\anN}{\mathcal{AN}}

\newcommand{\ma}{\mathcal{A}}

\newcommand{\ml}{\mathcal{L}}
\newcommand{\mn}{\mathcal{N}}

\newcommand{\mP}{\mathcal{P}}

\newcommand{\rmb}{\mathrm{b}}
\newcommand{\rmv}{\mathrm{v}}

\newcommand{\dist}{\mathrm{dist}}
\newcommand{\Div}{\mathrm{Div}}

\newcommand{\Gal}{\mathrm{Gal}}
\newcommand{\mult}{\mathrm{mult}}
\newcommand{\NE}{\mathrm{NE}}

\newcommand{\Norm}{\mathrm{Norm}}
\newcommand{\Nklt}{\mathrm{Nklt}}
\newcommand{\Nlc}{\mathrm{Nlc}}

\newcommand{\di}{\displaystyle}

 \equalenv{cor}{coro}
\equalenv{lem}{lemm}
\equalenv{defn}{defi}
\equalenv{rmk}{rema}

\equalenv{proposition}{prop}
\equalenv{corollary}{coro}

\equalenv{definition}{defi}
\equalenv{example}{exam}
\equalenv{remark}{rema}
\equalenv{conjecture}{conj}

\newtheorem{setup}{Set-up}

\title
[generalized canonical bundle formula]
{On a generalized canonical bundle formula for generically finite morphisms}

\alttitle{Sur une formule de bundle canonique g\'en\'eralis\'ee pour les morphismes g\'en\'eriquement finis}

\author{\firstname{Jingjun}  \lastname{HAN}}

\address{Johns Hopkins University \\ Department of Mathematics \\ Baltimore, MD 21218 (USA)}
\email{jhan@math.jhu.edu}

\author{\firstname{Wenfei}  \lastname{LIU}}

\address{Xiamen University  \\ School of Mathematical Sciences \\ Siming South Road 422 \\ Xiamen, Fujian 361005 (China)}

\email{wliu@xmu.edu.cn}

\thanks{We would like to thank Vladimir Lazi\'c for asking a question on a preliminary version of the paper and Stefano Filipazzi for sending us his PhD thesis \cite{Filipazzi19phdthesis}. Thanks also go to a referee who pointed out a gap in an earlier version of the paper and gave us several helpful comments and suggestions.  Zhiyu Tian kindly helped us with the French translation of the abstract of the paper. Part of the work was done when the authors were visiting other places: the first author would like to thank Professors Gang Tian and Zhiyu Tian for support and hospitality during a stay at Peking University in May 2019; the second author would like to thank Professor TEO Lee-Peng for support and hospitality while teaching at Xiamen University Malaysia Campus during the 2019/04 semester. The second author was partially supported by the NSFC (No.~11971399, No.~11771294).}

  
\altkeywords{paire généralisée, formule de bundle canonique, sous-adjonction}


\begin{document}
\begin{abstract}
We prove a canonical bundle formula for generically finite morphisms in the setting of generalized pairs (with $\RR$-coefficients). This complements Filipazzi's canonical bundle formula for morphisms with connected fibres. It is then applied to obtain a subadjunction formula for log canonical centers of generalized pairs. As another application, we show that the image of an anti-nef log canonical generalized pair has the structure of a numerically trivial   log canonical generalized pair. This readily implies a result of Chen--Zhang. Along the way we prove that the Shokurov type convex sets for anti-nef log canonical divisors are indeed rational polyhedral sets.
\end{abstract}

\begin{altabstract}
Nous prouvons une formule de bundle canonique pour des
morphismes g\'en\'eriquement finis dans le cadre de paires g\'en\'eralis\'ees
(avec $\RR$-coefficients). Cela compl\`ete la formule de bundle canonique
de Filipazzi pour les morphismes \`a fibres connect´ees. Elle est ensuite
appliqu\'ee pour obtenir une formule de sous-jonction pour les centres
log canoniques de paires g\'en\'eralis\'ees. Comme une autre application, nous
montrons que l'image d'une paire g\'en\'eralis\'ee canonique anti-nef log
a la structure d'une paire g\'en\'eralis\'ee canonique log num\'eriquement
triviale. Cela implique un r\'esultat de Chen-Zhang. En route, nous prouvons que les ensembles convexes de type de Shokurov pour les diviseurs log canoniques anti-nef sont en effet des
ensembles poly\'edriques rationnels.
\end{altabstract}


\maketitle

\section{Introduction}\label{sec: introduction}
To study a projective morphism $f\colon X\rightarrow Z$ between normal varieties, it is crucial to find relations between the canonical divisors $K_X$ and $K_Z$. One prominent example is Kodaira's canonical bundle formula for elliptic surfaces (\cite{Kodaira64AJM}). Nowadays, thanks to the work of several authors (\cite{Mori87, Kawamata97subadj, Kawamata98subadj, Fujino99, FM00, Ambro04JDG, Ambro05Compos, FG12, FG14}), the canonical bundle formula of Kodaira has a far reaching extension  to higher dimensions and to the log case as follows.
\begin{thm}\label{thm: cbf}
Let $f\colon X\rightarrow Z$ be a projective surjective morphism between normal varieties with connected fibres. Suppose that there is an effective $\RR$-divisor $B$ on $X$ such that $(X, B)$ is a projective log canonical pair and $K_X +B\sim_{\RR, f} 0$.  
Then there is an effective $\RR$-divisor $B_Z$ and an $\RR$-divisor $M_Z$ on $Z$ such that
\begin{enumerate}[leftmargin=*]
\item[(i)] $K_Z+B_Z+M_Z$ is $\RR$-Cartier and $K_X +B \sim_{\RR} f^*(K_Z+B_Z+M_Z)$, and
\item[(ii)] there is a projective birational morphism $\mu\colon \widetilde Z\rightarrow Z$ and a nef $\RR$-divisor $M_{\widetilde Z}$ such that $M_Z=\mu_* M_{\widetilde Z}$; writing $K_{\widetilde Z}+B_{\widetilde Z}  + M_{\widetilde Z} = \mu^*(K_Z+B_Z+M_Z)$, $(\widetilde Z,B_{\widetilde Z})$ is a sub-log canonical pair.
\end{enumerate} 
\end{thm}
\noindent Theorem~\ref{thm: cbf} allows one to investigate $(X, B)$ in terms of $(Z, B_Z+M_Z)$ which typically has lower dimension. If $(X, B)$ has klt singularities then it is possible to choose $M_Z$ such that $(Z, B_Z+M_Z)$ is still a klt pair (\cite[Theorem~0.2]{Ambro05Compos}). However, if $(X, B)$ is only assumed to be log canonical, then it is not known that $(Z, B_Z+M_Z)$ can land in the category of log canonical pairs; it depends on the  semiampleness conjecture about the divisor $M_{\widetilde Z}$ appearing in (ii) (\cite[Conjecture 7.13.1]{PS09}, see also \cite[Conjecture~3.9]{Fujino15}). 

It is the idea of Birkar--Zhang (\cite{BZ16}) that one can view $(Z, B_Z+M_Z)$ together with the nef $\RR$-divisor $M_{\widetilde Z}$ on the higher birational model $\widetilde Z$ as a \emph{generalized pair} and investigate its geometry as for usual pairs. Note that $\widetilde Z$ is allowed to be replaced by an even higher birational model and $M_{\widetilde Z}$ by its pull-back, but the newly obtained generalized pair is understood to be the same as the original one. It is thus convenient to treat $M_Z$ as the trace of a $\rmb$-$\RR$-divisor $\bfM_f$ on $Z$, that is, $M_Z=\bfM_{f, Z} $ (see Definition~\ref{def: b-div} for $\rmb$-divisors); then the generalized pair at hand can be formally written as $(Z, B_Z+\bfM_f)$. Now, by the second part of Theorem~\ref{thm: cbf} (ii), the generalized pair $(Z, B_Z+\bfM_f)$ has log canonical singularities. (We refer to Section~\ref{sec: g pair} for the definition of generalized pairs as well as their singularities.)

Since its inception, the notion of generalized pairs has proven useful in birational geometry and people are working intensively on the minimal model program of generalized pairs as well as its implications (\cite{BZ16, Ljh18Sarkisov,HM18weak,HL18wzd, Filipazzi18gcbf,Filipazzi18genpair, HanLi20accumulation,LT19,Filipazzi19phdthesis, Bir19, Li20fibrations, Chen20compgpair,HanLi20fujita, HanLiu20nonvanishing, Filipazzi20boundedness,FS20,Birkar20polarvar, HanLjh20,LMT20,ChenXue20compgpair,BirkarChen20,Birkar20connectedness,FS20connectedness,FW20connect}). There are also interesting interactions with other threads of the minimal model program such as the theory of quasi-log canonical pairs (\cite{Fujino18}) and the Generalized Nonvanishing Conjecture in \cite{LP20genabundance1}. We refer the reader to \cite{Birkar20genpair} for an exposition on some of the recent developments and open problems involving generalized pairs.

This paper treats the canonical bundle formula in the setting of generalized pairs. Filipazzi  proves a canonical bundle formula for morphisms with connected fibres from generalized pairs, under the assumption that the divisors involved have $\QQ$-coefficients (\cite[Theorem~1.4]{Filipazzi18gcbf} and \cite[Theorem 6]{Filipazzi19phdthesis}).  In this paper, we generalize his formula to any projective surjective morphism, at the same time relaxing the condition on the coefficients of the divisors.

\begin{thm}\label{thm: main}
	Let $\FF$ be either the rational number field $\QQ$ or the real number field $\RR$. Fix a quasi-projective scheme $S$. Let $f\colon X\to Z$ be a surjective morphism of normal varieties, projective over $S$. Suppose that there is an $\FF$-divisor $B$ and a nef $\rmb$-$\FF$-divisor $\bfM$ such that
\begin{itemize}[leftmargin=*]
\item $(X/S,B+\bfM)$ is a generalized pair over $S$;
\item $\bfM$ is an $\FF_{> 0}$-linear combination of nef$/S$ $\QQ$-Cartier $\rmb$-divisors;
\item $(X/S,B+\bfM)$ has  log canonical singularities over the generic point of $Z$;
\item $K_X+B+\bfM_X\sim_{\FF,f}0$.
\end{itemize}
Then there is an $\FF$-divisor $B_Z$ and a nef $\rmb$-$\FF$-divisor $\bfM_f$ on $Z$ such that $(Z/S, B_Z+\bfM_f)$ is a generalized pair over $S$ and $$K_X+B +\bfM_X\sim_{\FF} f^*(K_Z+B_Z +\bfM_{f,Z}).$$
Moreover, if $(X/S, B+\bfM)$ has log canonical (resp.~klt) singularities, then so does $(Z/S, B_Z+\bfM_{f})$.
\end{thm}

\noindent Theorem~\ref{thm: main} is a combination of the aforementioned result of Filipazzi with our canonical bundle formula for generically finite morphisms (Theorem~\ref{thm: gen finite}). We remark that Theorem~\ref{thm: gen finite} is a generalization of Fujino--Gongyo's result \cite[Theorem~3.1]{FG12} for log canonical pairs. Compared to \cite[Theorem~3.1]{FG12}, \cite[Theorem 1.4]{Filipazzi18gcbf} and \cite[Theorem 6]{Filipazzi19phdthesis}, there are two nontrivial ingredients in the proof of Theorem~\ref{thm: main}. The first ingredient is to reduce the theorem to the case with $\QQ$-coefficients; see Lemma~\ref{lem: R to Q}. Recall that the nonvanishing conjecture does not hold for generalized pairs (cf.~\cite{HanLiu20nonvanishing}), so the proof of \cite[Theorem~3.1]{FG12} does not work in our setting. The second ingredient is the construction of $B_Z$ and $\bfM_f$, which resembles the proof of \cite[Lemma~1.1]{FG12}; additional work is then done to show that the $\rmb$-$\QQ$-divisor $\bfM_f$ is nef.

One application of the canonical bundle formula is to establish subadjunction formulas for log canonical centers of generalized pairs of codimension larger than one. The subadjunction is proved in \cite[Theorems~1.5 and 6.7]{Filipazzi18gcbf} when the log canonical center is exceptional or when the underlying variety of the generalized pair is $\QQ$-factorial klt.  In Section~\ref{sec: subadj} we apply Theorem~\ref{thm: main} to obtain a subadjunction formula for general log canonical centers; we also prove the main result of \cite{CZ2013QofDPSJEMS} in the setting of  log canonical generalized pairs. 

Along the way we establish in Section~\ref{sec: polytope} that the Shokurov type convex set of anti-nef log canonical divisors is a rational polyhedral set. This is of independent interest.
\medskip

\noindent{\bf Notation and Conventions.} We work throughout over the complex number field $\CC$. A scheme means a separated scheme of finite type over $\CC$. The canonical divisors $K_X$ of normal varieties $X$ are always so taken that $f_*K_X=K_Y$ holds for a proper birational morphism $f\colon X\rightarrow Y$.
For a set $A$ of real numbers, we use $A_{> 0}$ to denote the subset $\{a\in A \mid a> 0\}$; for any $b\in \RR$, we denote $bA:=\{ba\mid a\in A\}$.


\section{Preliminary}
\subsection{Divisors}
We refer to \cite{Fujino17book} for divisors, $\QQ$-divisors, $\RR$-divisors and the Kleiman--Mori cones of curves. We give now the definition of $\rmb$-divisors, introduced by Shokurov; a nice discussion about this notion can be found in \cite{Corti07}.
\begin{defn}\label{def: b-div}
Let $X$ be a normal variety and let $\Div X$ be the free abelian group of Weil divisors on $X$. A \emph{$\rmb$-divisor} on $X$ is an element of the projective limit $${ \bfDiv} X = \lim_{Y\rightarrow X} \Div Y,$$ where the limit is taken over all the pushforward homomorphisms $\rho_*\colon \Div Y\rightarrow \Div X$ induced by proper birational morphisms $\rho\colon Y\rightarrow X$. In other words, a $\rmb$-divisor $\bfD$ on $X$ is a collection of Weil divisors $\bfD_Y$ on higher models of $X$ that are compatible under pushforward; the divisors $\bfD_Y$ are called the traces of $\bfD$ on the birational models $Y$.
\end{defn}

Let $\FF$ be either the rational number field $\QQ$ or the real number field $\RR$. Then a \emph{$\rmb$-$\FF$-divisor} is defined to be an element of $({\bfDiv} X)\otimes_\ZZ\FF$. The \emph{Cartier closure} of an $\FF$-Cartier $\FF$-divisor $D$ on $X$ is the $\rmb$-$\FF$-divisor $\overline D$ with trace $\overline D_Y=\rho^* D$ for any proper birational morphism $\rho\colon Y\rightarrow X$.
 A $\rmb$-$\FF$-divisor $\bfD$ on a normal variety $X$ is \emph{$\FF$-Cartier}  if $\bfD = \overline{D_Y}$ where  $D_Y$ is an $\FF$-Cartier $\FF$-divisor on a birational model over $X$; in this situation, we say $\bfD$ descends to $Y$. A $\rmb$-$\FF$-divisor is \emph{nef}  if it descends to a nef $\FF$-divisor on a birational model over $X$.
 
\subsection{Generalized pairs}\label{sec: g pair}
\begin{defn}\label{def: g pair}
Let $S$ be a scheme. A \emph{generalized pair} over $S$ consists of 
\begin{itemize}[leftmargin=*]
\item  a normal variety $X$ equipped with a projective morphism $X\rightarrow S$,
\item  an $\RR$-divisor $B$, and
\item a nef $\rmb$-$\RR$-divisor $\bfM$ on $X$, which is then the Cartier closure of a nef $\RR$-divisor $\widetilde M$ on a normal variety $\widetilde X$ equipped with a projective birational morphism $\rho\colon\widetilde X\rightarrow X$,
\end{itemize}
such that $K_X+B+M$ is $\RR$-Cartier, where $M=\rho_* \widetilde M$.  We denote the generalized pair by $(X/S, B+\bfM)$; if $S$ is a point then we drop $S$ from the notation.
\end{defn}

\begin{rmk} 
The $\RR$-divisor $M$ appearing in the above definition is nothing but the trace $\bfM_X$ of $\bfM$ on $X$.
\end{rmk}

Let $(X/S, B+\bfM)$ be a generalized pair.  For any prime divisor $E$ on a higher birational model $\widetilde X$, equipped with a proper birational morphism $\rho\colon \widetilde X\rightarrow X$, we write $\rho^*(K_X+B+\bfM_X) = K_{\widetilde X} + \widetilde B+\bfM_{\widetilde X}$, and define the \emph{discrepancy} of $E$ with respect to $(X/S, B+\bfM)$ as $$a_E(X/S, B+\bfM) = -\mult_E \widetilde B.$$ Then $(X/S, B+\bfM)$ is said to be \emph{sub-log canonical} (resp.~\emph{sub-Kawamata log terminal}) if  $a_E(X/S, B+\bfM)  \geq -1$ (resp.~$a_{E} (X/S, B+\bfM) >-1$) for any prime divisor $E$ over $X$. We omit the prefix "sub" everywhere if $B$ is effective. Log canonical (resp.~Kawamata log terminal) is often abbreviated to lc (resp.~klt), as usual. If there is a prime divisor $E$ over $X$ with $a_E(X/S, B+\bfM)\leq -1$ (resp.~$a_E(X/S, B+\bfM)< -1$) then its image in $X$ is called a \emph{non-klt center} (resp.~\emph{non-lc center}). A non-klt center that is not a non-lc center is called a \emph{lc center}. The union of all the non-klt centers (resp.~non-lc centers), denoted by $\Nklt(X/S, B+\bfM)$ (resp.~$\Nlc(X/S, B+\bfM)$), is called the \emph{non-klt locus} (resp.~\emph{non-lc locus}) of $(X/S, B+\bfM)$. 

\begin{rmk}
(a)   
Let $\rho\colon \widetilde X \rightarrow X$ be a log resolution such that $\bfM$ descends to $\widetilde X$ and $\Supp(\rho_*^{-1}B \cup \Exc(\rho))$ is a simple normal crossing divisor. Write  $K_{\widetilde X}+ \widetilde B +\bfM_{\widetilde X} = \rho^*(K_{X}+ B +\bfM_{X} )$. Then it is enough to look at the coefficients of $\widetilde B$ in order to determine the non-klt (resp.~non-lc) locus of $(X/S, B+\bfM)$. As a consequence, the non-klt  (resp.~non-lc) locus is a Zariski closed subset of $X$; see \cite[Lemma~2.3.20]{Fujino17book}.

(b) Our definition of non-klt (resp.~non-lc) locus is set-theoretic, without taking possible non-reduced scheme structures into account; compare \cite[2.3.11]{Fujino17book}.

\end{rmk}

\section{Shokurov type anti-nef polytopes}\label{sec: polytope}

In this section, we prove that the Shokurov type convex set of anti-nef log canonical divisors is a rational polyhedral set, following the treatment of \cite[Section~3]{Birkar11} and \cite{HL18wzd}; see Proposition~\ref{prop: rat polytope}. This is of independent interest and its corollary will be used in the proof of Theorem~\ref{thm: CZ}. 

\medskip

\begin{defn} Let $V$ be a finite dimensional vector space over $\RR$ with a specified basis $\{\rmv_i\}_i$. A subset $\mP\subset V$ is called a \emph{polyhedral set} if it is the intersection of finitely many closed halfspaces of $V$. Note that a polyhedral set can be empty or unbounded in this paper. The \emph{dimension} of a polyhedral set is defined to be the dimension of its affine hull.
A \emph{face} of a polyhedral set $\mP$ is the intersection of $\mP$ with a hyperplane such that $\mP$  lies in one of the two closed halfspaces defined by the hyperplane. We use $\partial \mP$ to denote the relative boundary of $\mP$, which is the union of all proper faces of $\mP$.   A polyhedral set is called a \emph{polytope} if it is bounded and is called \emph{rational} if its defining linear inequalities can be chosen to have rational coefficients with respect to the given basis $\{\rmv_i\}_i$. 
\end{defn}

We will stick to the following set-up.
\begin{setup}\label{su}
Let $S$ be a scheme and $X$ a $\QQ$-factorial projective/$S$ normal variety with lc singularities. Let $\{B_i\}_{i\in I}$ be a set of finitely many distinct prime divisors and $\{\bfM_j\}_{j\in J}$ a set of finitely many distinct nef/$S$ $\QQ$-Cartier $\rmb$-divisors on $X$.
Consider the $\RR$-vector space obtained by taking the outer direct sum
\[
V:=\left(\bigoplus_{i\in I} \RR B_i\right) \bigoplus\left(\bigoplus_{j\in J}\Rr \bfM_{j}\right).
\] 
The elements of $V$ are written as $(\Delta, \bfN)$ with $\Delta \in \bigoplus_i \RR B_i$ and $\bfN\in\bigoplus_j\Rr \bfM_{j}$.  One can define a norm $||\cdot||$ on $V$ as follows: for an element $(\Delta,\bfN)\in V$ with $\Delta=\sum_i x_i B_i$ and $\bfN = \sum_j y_j \bfM_j$, set
$$||(\Delta, \bfN)||:=\max_{i,j}\{|x_i|, |y_j|\}.$$ Let $V_{\geq 0} = \{(\sum_i x_i B_i,  \sum_j y_j \bfM_j) \in V \mid x_i\geq 0, y_j\geq 0 \text{ for all } i, j\}.$
Define 
\begin{equation*}\label{eq: lc polytope}
\ml(\{B_i\}_{i\in I}, \{\bfM_j\}_{j\in J}):=\left\{(\Delta, \bfN)\in V_{\geq 0} \mid\left(X/S,\Delta+\bfN\right) \text{ is generalized  lc}\right\}.
\end{equation*}
When the sets $\{B_i\}_{i\in I}, \{\bfM_j\}_{j\in J}$ are clear from context, we use $\ml_{IJ}$ to denote $\ml(\{B_i\}_{i\in I}, \{\bfM_j\}_{j\in J})$ in order to simplify notation. \end{setup}

Note that $\ml_{IJ}$ is not empty because $(0,0)\in \ml_{IJ}$. It is known that $\ml_{IJ}$ is a rational polyhedral set with respect to the basis $\{(B_i, 0)\}_i\cup\{(0,\bfM_j)\}_j$ of $V$ (\cite[Section~3.3]{HL18wzd}). We refine the description of $\ml_{IJ}$ as follows.
\begin{lem}\label{lem: L}
Let the notation be as in Set-up \ref{su}.
For any subset $J'\subset J$, possibly empty, let 
\[
p_{J'}\colon V\rightarrow \left(\bigoplus_{i\in I} \RR B_i\right)\bigoplus\left(\bigoplus_{j\in J'}\Rr \bfM_{j}\right).
\]
and $q_j\colon V\rightarrow \RR\bfM_j$ ($j\in J$) be the projections to the direct summands. Then the following holds.
\begin{enumerate}[leftmargin=*]
\item[(i)] The image of $\ml_{IJ}$ under $p_{J'}$ is $\ml_{IJ'}$, where $\ml_{IJ'} = \ml(\{B_i\}_{i\in I}, \{\bfM_j\}_{j\in J'})$ by the convention in Set-up~\ref{su}. 
\item[(ii)] For a given $j_0\in J$, the image $q_{j_0}(\ml_{IJ})$ is unbounded if and only if $\bfM_{j_0}$ descends to $X$, that is, there is a nef divisor $M_{j_0}$ on $X$ such that $\bfM_{j_0}$ is the Cartier closure $\overline{M_{j_0}}$.
\item[(iii)] Let $J_1=\{j\in J \mid  \bfM_j \text{ descends to } X\}$ and $J_2=J\backslash J_1$. Then the rational polyhedral set $\ml_{IJ_2}$ is bounded and
\begin{equation*}\label{eq: decomp 1}
\ml_{IJ}\cong  \left(\bigoplus_{j\in J_1} \RR_{\geq 0} \bfM_j\right)\bigoplus \ml_{IJ_2}.
\end{equation*}
\end{enumerate} 
\end{lem}
\begin{proof}
Fix a projective birational morphism $\rho\colon \widetilde X\rightarrow X$ such that all of the $\rmb$-divisors $\bfM_j$, $j\in J$ descend to $\widetilde X$, so the trace $\bfM_{j, \widetilde X}$ is a nef/$S$ divisor and $\bfM_j$ is the Cartier closure of  $\bfM_{j, \widetilde X}$. For $(\Delta, \bfN)\in V_{\geq 0}$, we can write
\begin{equation}\label{eq: pullback}
\rho^*(K_X+\Delta) = K_{\widetilde X} + {\widetilde \Delta}_0 \text{ and } \rho^*(K_X+\Delta + \bfN_X) = K_{\widetilde X} + {\widetilde \Delta} + \bfN_{\widetilde X}.
\end{equation}
where $ {\widetilde \Delta}_0$ and ${\widetilde \Delta}$ are uniquely determined $\RR$-divisors on $\widetilde X$ such that 
\[
\rho_* {\widetilde \Delta}_0 =\rho_* {\widetilde \Delta} =\Delta.
\]
The element $(\Delta, \bfN)$ lies in $\ml_{IJ}$ if and only if $(\widetilde X, {\widetilde \Delta})$ is sub-lc (see Section~\ref{sec: g pair}).

Since $\bfM_{j,\widetilde{X}}$ is nef over $S$, by the Negativity Lemma (\cite[Lemma~3.39]{KM98}), 
\begin{equation}\label{eq: neg lem}
\rho^{*}\bfM_{j,X}=\bfM_{j,\widetilde{X}}+E_j
\end{equation}
for some effective $\QQ$-divisor $E_j$, exceptional over $X$. Note that $E_j=0$ if and only if $\bfM_j$ descends to $X$. Substituting \eqref{eq: neg lem} into \eqref{eq: pullback} we infer that 
\begin{equation}\label{eq: Delta}
 {\widetilde \Delta}= {\widetilde \Delta}_0 + \sum_{j\in J} y_j E_j,
\end{equation}
where the $y_j$ ($j\in J$) are the (nonnegative) coefficients in the expression $\bfN=\sum_{j\in J} y_j \bfM_j$.

(i) For a subset $J'\subset J$ and $(\Delta, \bfN)\in \ml_{IJ}$ with $\bfN=\sum_j y_j\bfM_j$, the image of $(\Delta, \bfN)$ under $p_{J'}$ is $(\Delta, \bfN')$ with $\bfN' = \sum_{j\in J'}y_j \bfM_j$. Write
\begin{equation}\label{eq: pullback2}
\rho^*(K_X+\Delta + \bfN') = K_{\widetilde X} + {\widetilde \Delta}' + \bfN'_{\widetilde X}.
\end{equation}
Comparing this with \eqref{eq: pullback} and \eqref{eq: neg lem}, we obtain
\begin{equation}\label{eq: Delta'}
{\widetilde \Delta}= {\widetilde \Delta}' +\sum_{j\in J\backslash J'} y_j E_j.
\end{equation}
Since $(\widetilde X, {\widetilde \Delta})$ is a sub-lc pair and ${\widetilde \Delta}'\leq {\widetilde \Delta}$ by \eqref{eq: Delta'}, we infer that $(\widetilde X, {\widetilde \Delta}')$ is sub-lc. It follows that $(X/S, \Delta+\bfN')$ is lc by \eqref{eq: pullback2}, so $p_{J'}(\Delta, \bfN)\in \ml_{IJ'}$. On the other hand, any $(\Delta, \bfN')\in\ml_{IJ'}$ can be naturally viewed as an element of $\ml_{IJ}$ and it holds $p_{J'}(\Delta, \bfN') = (\Delta, \bfN')$. Thus we have proved that $p_{J'}(\ml_{IJ}) = \ml_{IJ'}$.

(ii) Fix $j_0\in J$. Suppose that the image $q_{j_0}(\ml_{IJ})$ is unbounded,  so $y_{j_0}$ is unbounded as $(\Delta, \bfN)\in\ml_{IJ}$ with $\bfN=\sum_j y_j \bfM_j$ varies. Since $(\widetilde X, {\widetilde \Delta})$ is sub-lc,  we have $\mult_E{\widetilde \Delta}\leq 1$ for any component $E$ of  ${\widetilde \Delta}$. In view of \eqref{eq: Delta}, this can happen only when $E_{j_0}=0$, so $\bfM_{j_0}$ descends to $X$.  

In the other direction, suppose that $\bfM_{j_0}$ descends to X. Then $(X/S, \Delta+\bfN)$ is lc if and only if $(X/S, \Delta+(\bfN + y_{j_0}'\bfM_{j_0}))$ is for any nonnegative number $y_{j_0}'$. It follows that the coefficient $y_{i_0}$ in $\bfN=\sum_{j\in J} y_j \bfM_j$ can be arbitrarily large, whence the unboundedness of $q_{j_0}(\ml(B,\bfM))$.

(iii) The boundedness of $\ml_{IJ_2}$ follows from (ii). The direct sum decomposition of (iii) follows from the observation that, $(\Delta, \bfN)\in V_{\geq 0}$ lies in  $\ml_{IJ}$ if and only if $(\Delta, (\bfN + \sum_{j\in J_1} y_j'\bfM_j))$ lies in $\ml_{IJ}$ for any nonnegative real numbers $y'_j$ with $j\in J_1$.
\end{proof}

A curve $\Gamma$ on $X$ is called \emph{extremal} if it generates an extremal ray $R$ of $\overline{\NE}(X/S)$ and if for some (equivalently, any) ample divisor $H$, we have $H\cdot\Gamma=\min\{H\cdot C\}$, where $C$ ranges over curves generating $R$. This definition of extremal curves is  slightly different from that in \cite[Section 3]{Birkar11}; we do not require the extremal ray to define a contraction.

The following existence of extremal curves as well as the bound on their lengths is crucial for the proofs of the main results in this section.
\begin{lem}\label{lem: length}
Let $S$ be a scheme. Let $(X/S, B+\bfM)$ be a lc generalized pair such that $X$ is $\QQ$-factorial klt. Then for any $(K_X+B+\bfM_X)$-negative extremal ray $R/S$, there exists an extremal curve $\Gamma$ generating $R$, and for any such $\Gamma$ it holds 
\[
-(K_X+B+\bfM_X)\cdot\Gamma\leq 2\dim X.
\]
\end{lem}
\begin{proof}
By \cite[Proposition 3.13]{HL18wzd}, there is a curve $C$ generating $R$ such that $(K_X+B+\bfM_X)\cdot C\ge-2\dim X$. For a given ample Cartier divisor $H$, the (nonempty) set $\{H\cdot C\}_C$, where $C$ ranges over curves generating $R$, consists of positive integers.  It follows that $\{H\cdot C\}$ attains its minimum as, say, $H\cdot\Gamma$. Then $\Gamma$ is an extremal curve we are looking for.
\end{proof}

Now we can state the main result of this section:
\begin{prop}\label{prop: rat polytope}
Let $X/S, \{B_i\}_{i\in I}, \{\bfM_j\}_{j\in J}, V$ and $\ml_{IJ}$ be as in Set-up~\ref{su}. Assume furthermore that $X$ is klt.  Let $\{R_t\}_{t\in T}$ be a family  of extremal rays of $\overline{\NE}(X/S)$. For any subset $\mP$ of $\ml_{IJ}$,  define
\begin{equation}\label{eq: AN}
\anN_{T}(\mP)=\{(\Delta,\bfN)\in \mP\mid -(K_X+\Delta+\bfN_{X})\cdot R_t\ge0 \text{ for any } t\in T\}.
\end{equation}
Suppose that $\anN_{T}(\ml_{IJ})$ is nonempty.   Then the following holds.
\begin{enumerate}[leftmargin=*]
\item[(i)] For a given $j_0\in J$, let $q_{j_0}\colon V\rightarrow \RR\bfM_{j_0}$ be the projection to the direct summand. Then the image $q_{j_0}(\anN_{T}(\ml_{IJ}))$ is unbounded if and only if $\bfM_{
j_0} \text{ descends to } X \text{ and } \bfM_{j_0,X}\cdot R_t =0 \text{ for any } t\in T$.
\item[(ii)] Set $J'=\{j\in J \mid  \bfM_j \text{ descends to } X \text{ and } \bfM_{j,X}\cdot R_t =0 \text{ for any } t\in T\}$ and $J''=J\backslash J'$. Then $\anN_T(\ml_{IJ''})$ is a (nonempty) rational polytope and 
\begin{equation}\label{eq: decomp AN}
\anN_T(\ml_{IJ}) \cong  \left(\bigoplus_{j\in J'} \RR_{\geq 0} \bfM_j\right)\bigoplus \anN_T\left( \ml_{IJ''}\right).
\end{equation}
In particular, $\anN_T(\ml_{IJ}) $ is a rational polyhedral set, bounded if and only if $J'=\emptyset$.
\end{enumerate}
\end{prop}
The proof of Proposition~\ref{prop: rat polytope} will be given in the end of this section after some preparation. Before that, we draw a consequence.
\begin{cor}\label{cor: rat decomp}
Let $X/S, \{B_i\}_{i\in I}, \{\bfM_j\}_{j\in J}$ and $\ml_{IJ}$ be as in Set-up~\ref{su}. Assume furthermore that $X$ is klt.  Suppose that there is an element $(B, \bfM)\in \ml_{IJ}$ such that $-(K_X+B+\bfM_{X})$ is nef over $S$. Then there are finitely many elements $(\Delta^{(k)}, \bfN^{(k)})\in \ml_{IJ}$ with rational coefficients, and $c_k \in \Rr_{>0}$ with $\sum_{k} c_k=1$ such that 
\begin{enumerate}[leftmargin=*]
\item[(i)] $K_X+B+\bfM_X=\sum_k c_k (K_X+\Delta^{(k)}+\bfN^{(k)}_{X})$, and
\item[(ii)] $-(K_X+\Delta^{(k)}+\bfN^{(k)}_{X})$ is nef for each $k$.
\end{enumerate}
\end{cor}
\begin{proof}
Let $\{R_t\}_{t\in T}$ be the family of all extremal rays of  $\overline{\NE}(X/S)$. Then $\anN_{T}(\ml_{IJ})$ contains $(B, \bfM)$, since $-(K_X+B+\bfM_{X})$ is nef over $S$. By Proposition~\ref{prop: rat polytope}, $\anN_{T}(\ml_{IJ})$ is a rational polyhedral set. It follows that there are elements $(\Delta^{(k)}, \bfN^{(k)})\in \anN_{T}(\ml_{IJ})\subset\ml_{IJ}$ with rational coefficients, and $c_k \in \Rr_{>0}$ such that $\sum_{k} c_k=1$,  $(B, \bfM)=\sum_k c_k (\Delta^{(k)}, \bfN^{(k)})$. It is then clear that the $(\Delta^{(k)}, \bfN^{(k)})$  satisfy  (i) and (ii) of the conclusion.
\end{proof}

Now we start the preparation for the proof of Proposition~\ref{prop: rat polytope}.
\begin{lem}\label{lem: extremal curve} Let $X/S, \{B_i\}_{i\in I}, \{\bfM_j\}_{j\in J}$ and $\ml_{IJ}$ be as in Set-up~\ref{su}. Assume furthermore that $X$ is klt. Fix an element $(B, \bfM) \in \ml_{IJ}$. Then the following holds.
\begin{enumerate}[leftmargin=*]
\item[(i)] There exists a real number $\alpha>0$, depending only on $(B, \bfM)$, such that  for any extremal curve $\Gamma$ with $(K_X+B+\bfM_X)\cdot \Gamma< 0$ we have $(K_X+B+\bfM_X)\cdot \Gamma< -\alpha$. 
\end{enumerate}
Let $\mP\subset \ml_{IJ}$ be a rational polyhedral set such that $(B, \bfM) $ is a relative interior point of $\mP$.
\begin{enumerate}[leftmargin=*]
\item[(ii)]  There exists a constant $\delta>0$, depending only on $(B, \bfM) $, such that for any extremal curve $\Gamma$ with $(K_X+B+\bfM_X)\cdot \Gamma< 0$ and for any $(\Delta,\bfN)\in\mP$ with $||(\Delta-B,\bfN-\bfM)||<\delta$, we have $(K_X+\Delta+\bfN_{X})\cdot \Gamma< 0$.
\item[(iii)] There is a constant $\beta>0$, depending only on $(B, \bfM) $, such that for any two elements $(\Delta', \bfN')\neq (\Delta'', \bfN'')\in \mP$ and for any extremal curve $\Gamma$ with $(K_X+B+\bfM_X)\cdot \Gamma< 0$, we have 
\[
\left|\frac{( \Delta'  + \bfN_{X}') -(\Delta''+ \bfN''_{X})}{||(\Delta'-\Delta'', \bfN' - \bfN'')||}\cdot\Gamma\right|< \beta.
\] 
\item[(iv)] Let $(\Delta', \bfN')\in\mP$. Then there is a constant $\delta'>0$, depending only on $(\Delta', \bfN')$ and $(B, \bfM)$, such that for any extremal curve $\Gamma$ with $(K_X+\Delta'+\bfN'_{X})\cdot \Gamma< 0$ and  $(K_X+B+\bfM_X)\cdot \Gamma< 0$, and for any $(\Delta'',\bfN'')\in\mP$ with $||(\Delta''-\Delta',\bfN''-\bfN')||<\delta'$, we have $(K_X+\Delta''+\bfN_{X}'')\cdot \Gamma< 0$.

\end{enumerate} 
\end{lem}
\begin{proof}
(i) By Lemma~\ref{lem: L}, $\ml_{IJ}$ is a rational polyhedral set, so there exist elements $(\Delta^{(1)}, \bfN^{(1)}), \dots, (\Delta^{(r)}, \bfN^{(r)})\in\ml_{IJ}$ with rational coefficients and real numbers $c_k$ ($1\leq k\leq r$) such that 
\[
\sum_{1\leq k\leq r} c_k=1 \text{ and } K_X+B+\bfM_X=\sum_{1\leq k\leq r} c_k(K_X+\Delta^{(k)}+\bfN^{(k)}_{X}).
\] Since $X$ is $\QQ$-factorial, there is a positive integer $N$ such that $N(K_X+\Delta^{(k)}+\bfN^{(k)}_{X})$ is Cartier for all $1\leq k\leq r$. 

Let $\Gamma$ be an extremal curve such that  $(K_X+B+\bfM_X)\cdot \Gamma< 0$. Then, for any $1\le k_0\le r$,
	\begin{equation}\label{eq: bound 1}
	\begin{split}
	0&>(K_X+B+\bfM_X)\cdot \Gamma\\
	&=\sum_{1\leq k\leq r} c_k(K_X+\Delta^{(k)}+\bfN^{(k)}_{X})
	\cdot \Gamma\\
	&\ge -\sum_{k\neq k_0}c_k\cdot 2\dim X+c_{k_0}(K_X+\Delta^{(k_0)}+\bfN^{(k_0)}_X)\cdot \Gamma\\
	&\ge -2\dim X+c_{k_0}(K_X+\Delta^{(k_0)}+\bfN^{(k_0)}_X)\cdot \Gamma.
	\end{split}
	\end{equation}
where the second inequality follows from Lemma~\ref{lem: length} and the last inequality follows from fact that $\sum_{k\neq k_0}c_k\leq \sum_{k}c_k =1$. By \eqref{eq: bound 1} and Lemma~\ref{lem: length} we obtain
\begin{equation}\label{eq: bound 2}
-2c_{k_0}\dim X \leq c_{k_0} (K_X+\Delta^{(k_0)}+\bfN^{(k_0)}_X) \cdot \Gamma< 2\dim X.
\end{equation}
Since $N(K_X+\Delta^{(k_0)}+\bfN^{(k_0)}_X)$ is Cartier, $c_{k_0}(K_X+\Delta^{(k_0)}+\bfN^{(k_0)}_X) \cdot \Gamma$ is contained in $(c_{k_0}/N)\ZZ$; combining this fact with \eqref{eq: bound 2} we infer that there are only finitely many possibilities for the numbers $c_{k_0}(K_X+\Delta^{(k_0)}+\bfN^{(k_0)}_X)\cdot \Gamma$, and hence for $(K_X+B+\bfM_{X})\cdot \Gamma$. It follows also that there is some $\alpha>0$, which depends only on $(B,\bfM)$, such that $(K_X+B+\bfM_X)\cdot \Gamma<-\alpha$.

(ii)  If $\dim \mP=0$ then $\mP = \{(B, \bfM)\}$ and there is nothing to prove.  So we can assume that $\dim \mP>0$. Since $(B, \bfM)$ is a relative interior point of $\mP$, the number $$d =\min\{\dist((B, \bfM), \partial\mP),1\} $$ is positive, where $\dist((B, \bfM), \partial\mP)$ denotes the distance between $(B, \bfM)$ and the relative boundary $\partial\mP$. Let $$\delta=\frac{\alpha d}{2\dim X},$$ where $\alpha$ is as in (i).  Suppose on the contrary that $(K_X+B+\bfM_X)\cdot \Gamma<0$ while $(K_X+\Delta+\bfN_{X})\cdot \Gamma\ge0$ for some $(\Delta,\bfN)\in \mP$ with $||(\Delta-B, \bfN-\bfM)||<\delta$. 
Then $(\Delta_t,\bfN_t):=(B, \bfM)+t\frac{(\Delta-B, \bfN-\bfM)}{||(\Delta-B, \bfN-\bfM)||}$ is contained in $\mP$ for $-d\leq t\leq ||(\Delta-B, \bfN-\bfM)||$. Note that the function $\varphi(t)=(K_X+\Delta_t+\bfN_{t,X})\cdot\Gamma$ is affine in $t$ and we have
\begin{equation}\label{eq: phi}
\begin{split}
&\varphi(||(\Delta-B, \bfN-\bfM)||) = (K_X+\Delta+\bfN_{X})\cdot \Gamma\ge0 \\
&\varphi(0) =  (K_X+B+\bfM_X)\cdot \Gamma<0.
\end{split}
\end{equation}
It follows that there is some $0<t_0\leq ||(\Delta-B, \bfN-\bfM)||$ such that $\varphi(t_0)=0$. Therefore, 
	\begin{multline*}
	(K_X+\Delta_{-d}+\bfN_{-d,X})\cdot \Gamma = \varphi(-d)
	=  \frac{d+t_0}{t_0}\varphi(0) 
	=\frac{d+t_0}{t_0}(K_X+B+\bfM_X)\cdot \Gamma \\
	 < \frac{d+\delta}{\delta} (K_X+B+\bfM_X)\cdot \Gamma
	 <-\left(1+\frac{2\dim X}{\alpha}\right) \alpha < -2\dim X,
	\end{multline*}	
where the equalities are by the definition and the affineness of $\varphi$, the first inequality is because $0<t_0<\delta$ and $ (K_X+B+\bfM_X)\cdot \Gamma<0$, and the second inequality is by the definitions of $\alpha$ and $\delta$. However, the above overall inequality contradicts Lemma~\ref{lem: length}.

(iii) Take $$\beta=\frac{4\dim X}{\delta},$$ where $\delta$ is as in (ii). First we deal with the special case where $(\Delta' , \bfN')=(B, \bfM)$. For any element $(\Delta'', \bfN'')\neq (B, \bfM)\in \mP$ and for any extremal curve $\Gamma$ with $(K_X+B+\bfM_X)\cdot \Gamma< 0$, we have by (ii) that 
\begin{equation}\label{eq: neg}
\left(K_X+B+\bfM_X+ \frac{\delta}{2} \frac{[( \Delta''  + \bfN_{X}'') -(B+\bfM_X)]}{||(\Delta''-B, \bfN'' - \bfM)||}\right)\cdot \Gamma <0.
\end{equation}
Plugging the bound on the lengths of extremal curves $(K_X+B+ \bfM_X)\cdot\Gamma\geq -2\dim X$ of Lemma~\ref{lem: length} into \eqref{eq: neg}, we obtain
\begin{equation}\label{eq: bound 3}
 \frac{[( \Delta''  + \bfN_{X}'') -(B+\bfM_X)]}{||(\Delta''-B, \bfN'' - \bfM)||}\cdot \Gamma < \frac{4\dim X}{\delta} =\beta.
\end{equation}
Replacing $( \Delta'' , \bfN'') $ in \eqref{eq: bound 3} with $(B, \bfM) - \epsilon((\Delta'', \bfN'') - (B, \bfM))$, which still lies in $\mP$ for $0<\epsilon\ll 1$,
we obtain 
\begin{equation}\label{eq: bound 4}
-  \frac{[( \Delta''  + \bfN_{X}'') -(B+\bfM_X)]}{||(\Delta''-B, \bfN'' - \bfM)||}\cdot \Gamma < \beta.
\end{equation}
Combining \eqref{eq: bound 3} and \eqref{eq: bound 4} we have the required bound on the absolute value of the intersection number
\[
 \left|  \frac{[( \Delta''  + \bfN_{X}'') -(B+\bfM_X)]}{||(\Delta''-B, \bfN'' - \bfM)||}\cdot \Gamma\right| <\beta.
\]

Now for any two distinct elements $(\Delta' , \bfN'), (\Delta'' , \bfN'') \in \mP$, since $(B, \bfM)$ is a relative interior point of $\mP$, there exists $(\Delta''', \bfN''')\neq(B, \bfM)\in \mP$ such that 
\[
(\Delta''' -B, \bfN'''-\bfM) = \epsilon(\Delta'' -\Delta', \bfN''-\bfN')
\]
for some $0<\epsilon\ll 1$,
so we reduce to the special case which has already been handled.

(iv) By (i) there is a constant $\alpha'>0$, depending only on $(\Delta', \bfN')$, such that if $(K_X+\Delta'+\bfN_{X}')\cdot \Gamma< 0$ for some extremal curve $\Gamma$ then $(K_X+\Delta'+\bfN_{X}')\cdot \Gamma< -\alpha'$. Now set $$ \displaystyle \delta'  = \frac{\alpha'}{\beta},$$  where $\beta$ is as in (iii). Then  for any extremal curve $\Gamma$ with $(K_X+\Delta'+\bfN'_{X})\cdot \Gamma< 0$ and  $(K_X+B+\bfM_X)\cdot \Gamma< 0$, and for any $(\Delta'',\bfN'')\in\mP$ with $||(\Delta''-\Delta',\bfN'-\bfN')||<\delta'$,  we have
\begin{align*}
&(K_X+\Delta''+\bfN''_{X}) \cdot \Gamma \\
 =& (K_X+\Delta'+\bfN'_{X}) \cdot \Gamma + [(\Delta''+ \bfN''_X) -(\Delta'+ \bfN'_X)]\cdot \Gamma\\
\leq& (K_X+\Delta'+\bfN'_{X}) \cdot \Gamma + ||(\Delta'', \bfN'') -(\Delta', \bfN')||\cdot \left|\frac{[(\Delta''+ \bfN''_X) -(\Delta'+ \bfN'_X)]}{||(\Delta'', \bfN'') -(\Delta', \bfN')||}\cdot \Gamma\right | \\
<& - \alpha' + \delta'\beta= 0.
\end{align*}
where we use (iii) for the second inequality. 
\end{proof}

\begin{lem}\label{lem: bounded rational}
Let $X/S, \{B_i\}_{i\in I}, \{\bfM_j\}_{j\in J}, V$ and $\ml_{IJ}$ be as in Set-up~\ref{su}. Assume furthermore that $X$ is klt.  For any family $\{R_t\}_{t\in T}$ of extremal rays of $\overline{\NE}(X/S)$ and for any  rational polytope $\mP$ contained in $\ml_{IJ}$,  the subset $\anN_{T}(\mP)$, defined as in Proposition~\ref{prop: rat polytope}, is a rational polytope.
\end{lem}
\begin{proof}
The proof proceeds by induction on $\dim\mP$. It is clear if $\dim\mP=0$ or if $\anN_{T}(\mP)=\emptyset$. So we assume that $\dim\mP>0$ and $\anN_{T}(\mP)\neq \emptyset$. By induction we can assume that $\anN_{T}(\mP)$ is not contained in any proper face of $\mP$, and thus $\anN_{T}(\mP)$ contains a point that is simultaneously a relative interior point of $\anN_{T}(\mP)$ and of $\mP$. We fix this point once and for all, and denote it by $(B, \bfM)$. 

\noindent\textcircled{a}  By dropping those $t$ such that $\ma\mn_{\{t\}}(\mP)= \mP$, we may assume that for each $t\in T$, there is some $(\Delta,\bfN)\in \mP$ such that $(K_X+\Delta+\bfN_{X})\cdot R_t >0$.

\noindent\textcircled{b}  We claim that for any $t\in T$ there is some $(\Delta,\bfN)\in\mP$ such that $$(K_X+\Delta+\bfN_{X})\cdot R_t<0.$$ Otherwise, there is some $t_0\in T$ such that 
$(K_X+\Delta+\bfN_{X})\cdot R_{t_0}\geq 0$ for any $(\Delta,\bfN)\in \mP$.
 Then 
\[
\ma\mn_{\{t_0\}}(\mP) = \{(\Delta,\bfN)\in\mP\mid (K_X+\Delta+\bfN_{X})\cdot R_{t_0}=0\}.
\]
The inclusion $\{t_0\}\subset T$ implies the inclusion in the reversed direction $\ma\mn_{T}(\mP)\subset\ma\mn_{\{t_0\}}(\mP)$. But  $\ma\mn_{\{t_0\}}(\mP)$ is a proper face of $\mP$ due to the assumption \textcircled{a} above, contradicting the choice of $\mP$.

\noindent\textcircled{c} We claim that,  for any $t\in T$,  $$(K_X+B+\bfM_X)\cdot R_t <0$$ holds. Otherwise we have $(K_X+B+\bfM_X)\cdot R_{t_0} = 0$ for some $t_0\in T$. By  \textcircled{b} there is some $(\Delta,\bfN)\in \mP$ such that $(K_X+\Delta+\bfN_{X})\cdot R_{t_0}<0$, and thus 
\begin{equation}\label{eq: pos}
(B+\bfM_X-\Delta-\bfN_X)\cdot R_{t_0} >0.
\end{equation}
 Let $(\Delta', \bfN') = (B, \bfM) - \epsilon (\Delta-B , \bfN-\bfM)$ where $\epsilon>0$ is sufficiently small. Since $(B, \bfM)$ lies in the relative interior of $\anN_{T}(\mP)$, the vector $(\Delta', \bfN')$ still lies in $\anN_{T}(\mP)$. On the other hand, by \eqref{eq: pos}, we have
\begin{multline*}
(K_X+\Delta' +\bfN'_{X} )\cdot R_{t_0} = (K_X+B +\bfM_{X} )\cdot R_{t_0} + \epsilon (B+\bfM_X-\Delta  -\bfN_X) \cdot R_{t_0}  >0,
\end{multline*}
which is a contraction to the fact that $(\Delta', \bfN') \in\anN_{T}(\mP)$.

\medskip

As a consequence of \textcircled{b}, there is an extremal curve $\Gamma_t$ generating $R_t$ for any $t\in T$ by Lemma~\ref{lem: length}. By \textcircled{b} we have $(K_X+B+\bfM_X)\cdot \Gamma_t <0$. Now we can apply Lemma~\ref{lem: extremal curve} (iv)
to conclude that for each $(\Delta, \bfN)\in \anN_{T}(\mP)$, there is a constant $\delta(\Delta, \bfN)>0$, depending only on $(\Delta, \bfN)$ and $(B, \bfM)$, such that for any extremal curve $\Gamma_t$ ($t\in {T}$) with $(K_X+\Delta+\bfN_{X})\cdot \Gamma_t< 0$, and for any $(\Delta',\bfN')\in\mP$ with $||(\Delta'-\Delta,\bfN'-\bfN)||<\delta(\Delta, \bfN)$, we have $(K_X+\Delta'+\bfN_{X}')\cdot \Gamma_t< 0$. This yields an open cover of $\anN_{T}(\mP)$ as $(\Delta, \bfN)$ varies. Note that, being a closed convex subset of $\mP$, $\anN_{T}(\mP)$ is compact. Thus one can find finitely many elements  $(\Delta^{(1)},\bfN^{(1)}),\ldots,(\Delta^{(n)}, \bfN^{(n)})$ in $\anN_{T}(\mP)$ and real numbers $\delta_1,\ldots,\delta_n>0$ such that the following holds.
\begin{enumerate}[leftmargin=*]
		\item[(a)] $\anN_{T}(\mP)$ is covered by $$\Bb_k=\{(\Delta,\bfN)\in \mP \mid ||(\Delta-\Delta^{(k)}, \bfN-\bfN^{(k)})||<\delta_k\}.$$
		\item[(b)] If $(\Delta,\bfN)\in\Bb_k$ with $(K_X+\Delta+\bfN_{X})\cdot \Gamma_t>0$ for some $t\in {T}$, then $$(K_X+\Delta^{(k)}+\bfN^{(k)}_{X})\cdot \Gamma_t=0.$$ 
	\end{enumerate}
	Let $T_k=\{t\in {T}\mid (K_X+\Delta+\bfN_{X})\cdot \Gamma_t>0 \text{ for some }(\Delta,\bfN)\in \Bb_k\}.$  
	\begin{enumerate}[leftmargin=*]
		\item[(c)] $\di {T}=\bigcup_k {T_k}$ by the assumption \textcircled{a} made at the beginning of the proof, so $$\ma\mn_{T}(\mP) = \bigcap_{1\leq k\leq n} \ma\mn_{{T_k}}(\mP).$$ 
		\item[(d)] 
		As a consequence of (b),
		 $(K_X+\Delta^{(k)}+\bfN^{(k)}_{X})\cdot \Gamma_t=0$ for any $t\in {T_k}$.
	\end{enumerate}
	
\medskip

It suffices to prove that $\anN_{T_k}(\mP)$ is a rational polytope for each $k$. Replacing $\anN_{T}(\mP)$ with $\anN_{T_k}(\mP)$, we can assume that there is an element $(\Delta^{(0)}, \bfN^{(0)})\in \anN_{T}(\mP)$ such that $(K_X+\Delta^{(0)}+\bfN^{(0)}_{X})\cdot \Gamma_t=0$ for any $t\in {T}$. Note that $(\Delta^{(0)}, \bfN^{(0)})$ belongs to the following set
\[
\mP' = \{(\Delta,\bfN)\in\mP \mid (K_X+\Delta+\bfN_{X})\cdot \Gamma_{t}=0 \text{ for any } t\in {T} \}
\]
which is contained in $\anN_{T}(\mP)$. Although the index set $T$ may be
infinite, by simple linear algebra,  $\mP'$ is cut out from $\mP$ by finitely may affine hyperplanes of the form $(K_X+\Delta+\bfN_{X})\cdot \Gamma_{t}=0$. Thus $\mP'$ is a (nonempty) rational polytope and we can find $(\Delta', \bfN')\in\mP'\subset\anN_{T}(\mP)$ with rational coefficients such that $(K_X+\Delta'+\bfN'_{X})\cdot \Gamma_{t}=0$ holds  for any $t\in {T}$. 

\medskip

\noindent{\bf Claim. } $\anN_{T}(\mP)$ is the convex hull of $(\Delta', \bfN')$ and $\anN_{T}(\partial \mP)$.
\begin{proof}[Proof of the claim.]
For any element $(\Delta'', \bfN'')\neq (\Delta', \bfN')$ in $\anN_{T}(\mP)$, let $L$ be the ray from $(\Delta', \bfN')$ in the direction of $(\Delta'', \bfN'')$ and let $(\Delta''', \bfN''')\in\partial\mP$ be the point where $L$ leaves $\mP$. Since 
\[
(K_X+\Delta'+\bfN'_{X})\cdot \Gamma_{t}=0 \text{ and } (K_X+\Delta''+\bfN''_{X})\cdot \Gamma_{t}\leq 0,
\]
the segment between $(\Delta', \bfN')$ and $(\Delta''', \bfN''')$ is contained in $\anN_{{T}}(\mP)$; of course, this segment contains the given point $(\Delta'', \bfN'')$.  This finishes the proof the claim.
\end{proof}

Now we write $\partial\mP = \bigcup_{1\leq l\leq s}\mP_l$, where the $\mP_l$ are the codimension one faces of $\mP$, then $$\anN_{T}(\partial \mP) = \bigcup_{1\leq l\leq s}\anN_{T}(\mP_l).$$ By induction, the $\anN_{T}(\mP_l)$, if nonempty, are rational polytopes. By the claim above, $\anN_{T}(\mP)$ is a rational polytope.
\end{proof}

\begin{proof}[Proof of Proposition~\ref{prop: rat polytope}] 
(i)  Suppose that $q_{j_0}(\anN_{T}(\ml_{IJ}))\subset\RR\bfM_{j_0}$ is unbounded. Since $\anN_{T}(\ml_{IJ})$ is a subset of $\ml_{IJ}$ by definition, the image $q_{j_0}(\ml_{IJ})$ is also unbounded. By Lemma~\ref{lem: L} (ii), $\bfM_{j_0}$ descends to a nef divisor, say $M_{j_0}$, on $X$.

\medskip

\noindent{\bf Claim.} $M_{j_0} \cdot R_t =0$ for any $t\in T$.
\begin{proof}[Proof of the claim.]
Otherwise $M_{j_0} \cdot R_{t_0}>0$ for some $t_0\in T$. Since the image $q_{j_0}(\anN_{T}(\ml_{IJ}))\subset\RR\bfM_{j_0}$ is unbounded, there is an element $(\Delta, \bfN)\in\anN_{T}(\ml_{IJ})$ such that $\bfN= \sum_j y_j \bfM_j$ with $y_{j_0}>2r\dim X$, where $r$ is the Cartier index of $M_{j_0}$.
For a one-cycle $\Gamma_{t_0}$ generating $R_{t_0}$ we have
\[
0\geq (K_X + \Delta +\bfN_X) \cdot\Gamma_{t_0}= (K_X+ \Delta + \sum_{j\neq j_0} y_j \bfM_{j,X})\cdot\Gamma_{t_0} + y_{j_0} M_{j_0}\cdot \Gamma_{t_0}.
\]
It follows that 
\begin{equation}\label{eq: bound 5}
(K_X+ \Delta + \sum_{j\neq j_0} y_j \bfM_{j,X})\cdot\Gamma_{t_0} \leq - y_{j_0} M_{j_0}\cdot \Gamma_{t_0}  < -(2r\dim X)(M_{j_0}\cdot \Gamma_{t_0}).
\end{equation}
Note that $(X, \Delta + \sum_{j\neq j_0} y_j \bfM_j)$ is lc and $R_{t_0}$ is a $(K_X+ \Delta + \sum_{j\neq j_0} y_j \bfM_{j,X})$-negative extremal ray by \eqref{eq: bound 5}. So we can choose $\Gamma_{t_0}$ to be an extremal curve by Lemma~\ref{lem: length}, but then $rM_{j_0}\cdot \Gamma_{t_0} \geq 1$, and \eqref{eq: bound 5} contradicts the bound on the lengths of extremal curves in Lemma~\ref{lem: length}.
\end{proof}

For the other direction of implication, suppose that $\bfM_{j_0}$ descends to $X$ and $\bfM_{j_0, X}\cdot R_t=0$ for any $t\in T$. Then one sees easily that $(\Delta, \bfN)\in \anN_T(\ml_{IJ})$ if and only if $(\Delta, \bfN+ y'_{j_0} \bfM_{j_0})\in \anN_T(\ml_{IJ} )$ for any $y'_{j_0}>0$. It follows that the coefficient of $\bfM_{j_0}$ in the elements of $\anN_T(\ml_{IJ})$ can be arbitrarily large.

(ii) The boundedness of $\anN_T(\ml_{IJ''})$ and the decomposition \eqref{eq: decomp AN} follow easily from (i) and its proof. Since $\anN_T(\ml_{IJ''})$ is bounded, we can take a rational polytope $\mP$ such that $\anN_T(\ml_{IJ''})\subset\mP\subset\ml_{IJ''}$. Then we have
\[
\anN_T(\ml_{IJ''}) = \anN_T(\anN_T(\ml_{IJ''})) \subset \anN_T(\mP) \subset \anN_T(\ml_{IJ''}).
\]
It follows that  $\anN_T(\ml_{IJ''}) =  \anN_T(\mP)$, which is a rational polytope by Lemma~\ref{lem: bounded rational}.
\end{proof}

\section{The generalized canonical bundle formula}
In this section we first prove a lemma decomposing relatively trivial log canonical divisors with $\RR$-coefficients as the sum of those with $\QQ$-coefficients; this allows us to consider only generalized pairs with $\QQ$-coefficients. The basic idea has been explained in \cite[2.11]{Fli92} and \cite[Proposition~2.21]{KollarMMP13}; our proof follows closely that of \cite[Lemma~11.1]{Fujino18}.

\begin{lem}\label{lem: R to Q}
Let $(X/S, B + \bfM)$ be a generalized  pair over a scheme $S$ such that $\bfM$ is an $\RR_{>0}$-linear combination of nef$/S$ $\QQ$-Cartier divisors. Let $f\colon X\rightarrow Z$ be a projective surjective morphism between normal varieties over $S$ such that
$$K_X + B+\bfM_X \sim_{\RR, f} 0.$$
Then for any given $\epsilon >0$ there are finitely many positive real numbers $c_{\alpha}$ satisfying $\sum_{\alpha} c_{\alpha}=1$ and generalized pairs $(X/S, B^{(\alpha)} + \bfM^{(\alpha)})$ with $\QQ$-coefficients such that $K_X+B + \bfM_X = \sum_{{\alpha}} c_{\alpha}(K_X + B^{(\alpha)} +\bfM^{(\alpha)}_X)$, and the following holds for each $\alpha$.
\begin{enumerate}[leftmargin=*]
\item[(i)] $K_X + B^{(\alpha)} +\bfM^{(\alpha)}_X\sim_{\QQ, f}0$;
\item[(ii)] $|\mult_D(B^{(\alpha)}-B)|$ is at most $\epsilon$ for any prime divisor $D$ on $X$, and is zero if $\mult_D(B)$ is rational;
\item[(iii)] $B^{(\alpha)}$ is effective if $B$ is;
\item[(iv)] there are equalities:
\begin{itemize}
\item $\Nlc(X/S, B+\bfM) = \Nlc(X/S, B^{(\alpha)} + \bfM^{(\alpha)})$,
\item $\Nklt(X/S, B+\bfM) = \Nklt(X/S, B^{(\alpha)}+\bfM^{(\alpha)})$.
\end{itemize}
\end{enumerate}
\noindent In particular, if $(X/S, B+\bfM)$ is lc (resp.~klt), then so is $(X/S, B^{(\alpha)} + \bfM^{(\alpha)})$.
\end{lem}
\begin{proof}
By assumption we have $\bfM = \sum_{1\leq j\leq r} \mu_j \bfM_j$ where $\mu_j\in \RR_{> 0}$ and $\bfM_j$ are nef$/S$ $\QQ$-Cartier $\rmb$-divisors. Let $\rho\colon \widetilde X\rightarrow X$ be a log resolution such that the $\rmb$-divisors $\bfM_j$ all descend to $\widetilde X$, so the traces satisfy $\bfM_{\widetilde X} = \sum_j \mu_j \bfM_{j, \widetilde X}$ with $\bfM_{j, \widetilde X}$ Cartier and nef$/S$. Write $K_{\widetilde X} + \widetilde B+\bfM_{\widetilde X} = \rho^*(K_X + B+\bfM_X)$.
The condition $K_X + B+\bfM_X \sim_{\RR, f} 0$  means that there is an $\RR$-Cartier $\RR$-divisor $D=\sum_{1\leq k\leq m} d_k D_k$ on $Z$ such that $K_{X} +B+\bfM_{X}\sim_{\RR} f^* D$, where $d_k\in \RR$ and $D_k$ are Cartier divisors for each $1\leq k\leq m$. Pulling back to $\widetilde X$ by $\rho$, we obtain
$K_{\widetilde X} +\widetilde B+\bfM_{\widetilde X}\sim_{\RR} \widetilde f^* D$, where $\widetilde f = f\circ \rho$. Thus there are $a_l\in\RR$ and rational functions $\varphi_l$ on $\widetilde X$ such that
\[
K_{\widetilde X} +\widetilde B+\bfM_{\widetilde X} + \sum_{1\leq l\leq p} a_l(\varphi_l) = \sum_{1\leq k\leq m} d_k\widetilde f^* D_k
\]
where $(\varphi_l)$ denotes the divisor of $\varphi_l$.

Write $\widetilde B = \sum_{1\leq i\leq u} b_i {\widetilde B_i}$ where the ${\widetilde B_i}$ are distinct prime divisors. We may assume that $b_i\in \RR\setminus\QQ$ for $1\leq i\leq q$ and $b_i\in \QQ$ for $i\geq q+1$. Now consider the following linear map
\[
\Phi \colon \RR^{m+p+q+r} \rightarrow \Div(\widetilde X) \otimes_{\ZZ} \RR
\]
defined by 
\begin{multline*}
\Phi(x_1, \dots, x_{m+p+q+r}) = 
 \sum_{1\leq k\leq m} x_{k} \widetilde f^*D_k - \sum_{1\leq l\leq p}  x_{m+l} (\varphi_l) \\
 - \sum_{1\leq i \leq q} x_{m+p+i} {\widetilde B_i}  - \sum_{1\leq j\leq r} x_{m+p+q+j} \bfM_{j, \widetilde X}.
\end{multline*}
We note that $\Phi$ is defined over $\QQ$, and thus $\ma:=\Phi^{-1}(K_{\widetilde X} + \sum_{i\geq q+1} b_i {\widetilde B_i})$, containing the point $P=(d_1,\dots,d_m,a_1,\dots,a_p,b_1,\dots,b_q, \mu_1, \dots, \mu_r)$, is a nonempty affine space defined over $\QQ$. Therefore, there exist $P_1,\dots, P_n\in\ma\cap \QQ^{m+p+q+r}$ and $c_1,\dots, c_n\in \RR_{>0}$ such that 
\begin{equation}\label{eq: decomp 2}
\sum_{1\leq\alpha\leq n} c_{\alpha}=1\text{ and }\sum_{1\leq\alpha\leq n} c_{\alpha} P_{\alpha}   = P.
\end{equation}
Moreover, we can choose $$P_{\alpha}=(d_1^{(\alpha)},\dots,d_m^{(\alpha)},a_1^{(\alpha)},\dots,a_p^{(\alpha)},b_1^{(\alpha)},\dots,b_q^{(\alpha)}, \mu_1^{(\alpha)}, \dots, \mu_r^{(\alpha)})$$ sufficiently close to $P$, and obtain the corresponding generalized pairs $(\widetilde X, \widetilde B^{(\alpha)} + \bfM^{(\alpha)})$ with 
\[
\widetilde B^{(\alpha)} = \sum_{1\leq i\leq 1} b_i^{(\alpha)}\widetilde B_i  + \sum_{i\geq q+1} b_i {\widetilde B_i}\text{ and } \bfM^{(\alpha)}=\sum_{1\leq j\leq r} \mu_j^{(\alpha)} \bfM_j.
\]
By \eqref{eq: decomp 2}, we have $\widetilde B=\sum_{1\leq\alpha\leq n} c_{\alpha} \widetilde B^{(\alpha)}$ and $\bfM=\sum_{1\leq\alpha\leq n} c_{\alpha} \bfM^{(\alpha)}$.

The equation $\Phi(P_{\alpha}) = K_{\widetilde X} + \sum_{i\geq q+1} b_i {\widetilde B_i}$ implies that 
$$K_{\widetilde X}+ \widetilde B^{(\alpha)} + \bfM^{(\alpha)}_X \sim_{\QQ, \widetilde f}0.$$
By construction $\mult_{{\widetilde B_i}} (\widetilde B^{(\alpha)}) = \mult_{{\widetilde B_i}}(\widetilde B) $ for $i\geq q+1$. 
Since the $P_\alpha$ are sufficiently closed to $P$, we have $$\Nlc(\widetilde X/S, \widetilde B+\bfM) = \Nlc(\widetilde X/S, \widetilde B^{(\alpha)} + \bfM^{(\alpha)}),$$
and
$$\Nklt(\widetilde X/S, \widetilde B+\bfM) = \Nklt(\widetilde X/S, \widetilde B^{(\alpha)}+\bfM^{(\alpha)}).$$
Setting $B^{(\alpha)} = \rho_*\widetilde B^{(\alpha)}$, one readily verifies that $(X/S, B^{(\alpha)}+\bfM^{(\alpha)})$ are generalized pairs satisfying all the properties in the conclusion of the lemma.
\end{proof}

We need a simultaneous partial resolution of a generically finite morphism.
\begin{lem}\label{lem: factorize}
Let $f\colon X \rightarrow Z$ be a generically finite projective surjective morphism between normal varieties. Then there is a commutative diagram of projective surjective morphisms between normal varieties
\begin{center}
\begin{tikzcd}
\widetilde X \arrow[r, "\rho"] \arrow[d, "\widetilde f"]& X\arrow[d, "f"] \\
\widetilde Z \arrow[r, "\mu"] & Z
\end{tikzcd} 
\end{center}
such that the following holds.
\begin{enumerate}[leftmargin=*]
\item[(i)] The morphisms $\rho$ and $\mu$ are birational and $\widetilde f$ is finite, and
\item[(ii)] $\widetilde X$ and $\widetilde Z$ have at most quotient singularities.
\end{enumerate}
\end{lem}
\begin{proof}
First suppose that $X$ and $Z$ are both complete. Since $f\colon X\rightarrow Z$ is generically finite, we have a finite extension of function fields $\CC(X)/\CC(Z)$. Let $L$ be the Galois closure of $\CC(X)$ in the algebraic closure $\overline{\CC(Z)}$, and $Y$ the normalization of $Z$ in $L$. By construction, $Y$ is a projective normal variety with function field $\CC(Y) = L$, and there is  a Galois finite morphism $g\colon Y\rightarrow Z$ with Galois group $G=\Gal(\CC(Y)/\CC(Z))$. 

By the argument of \cite[Theorem~6.1]{Cutkosky03}, we can find a $G$-equivariant resolution $\widetilde Y\rightarrow Y$ such that the rational map from $\widetilde Y$ to $X$ induced by the field extension $\CC(Y)/\CC(X)$ is a morphism.  Let $H=\Gal(\CC(Y)/\CC(X))$ be the Galois group. Now we simply take $\widetilde X = \widetilde Y/H$ and $\widetilde Z = \widetilde Y/ G$, and let $\widetilde f\colon \widetilde X\rightarrow \widetilde Z$ be the natural finite morphism. 

In general, when $Z$ is not necessarily complete, we can fit the morphism $f\colon X\rightarrow Z$ into the following commutative diagram
\begin{center}
\begin{tikzcd}
 X \arrow[r, hook,"\iota"] \arrow[d,  "f"]& \bar X\arrow[d, "\bar f"] \\
 Z \arrow[r, hook, "i"] & \bar Z
\end{tikzcd} 
\end{center}
where $\iota$ and $i$ are open embeddings into complete normal varieties, and apply the previous arguments to obtain
\begin{center}
\begin{tikzcd}
\widetilde X \arrow[r, "\bar \rho"] \arrow[d, "\widetilde f"]& \bar X\arrow[d, "\bar f"] \\
\widetilde Z \arrow[r, "\bar\mu"] & \bar Z
\end{tikzcd} 
\end{center}
with $\bar \rho$ and $\bar\mu$ satisfying (i) and (ii). Now restricting to $Z\subset \bar Z$ and its inverse images gives us the commutative diagram we wanted.
\end{proof}
\begin{rmk}\label{rmk: fibre prod}
By the universal property of fibre products, one sees that $\widetilde X$ is isomorphic to the normalization of the main component of $X\times_{Z} \widetilde Z$.
\end{rmk}

We prove some compatibility of pullback of divisors under birational morphisms with pushforward of divisors under generically finite morphisms.
\begin{lem}\label{lem: push&pull}
Let $\FF$ be either the rational field $\QQ$ or the real number field $\RR$. Consider a commutative diagram of projective surjective morphisms between normal varieties
\begin{center}
\begin{tikzcd}
 \widehat X \arrow[r, "\widetilde \rho"] \arrow[d, "\widehat f"]& \widetilde X\arrow[d, "\widetilde f"] \\
 \widehat Z \arrow[r, "\widetilde \mu"] & \widetilde Z
\end{tikzcd} 
\end{center}
such that
\begin{itemize}[leftmargin=*]
\item $\widetilde \rho$ and $\widetilde \mu$ are birational morphisms, and
\item $\widetilde f$ is finite, so $\widehat f$ is generically finite.
\end{itemize}
Then for any $\FF$-Cartier $\FF$-divisor (resp.~Cartier) $D$ on $\widetilde X$, the following holds.
\begin{enumerate}[leftmargin=*]
\item[(i)] The pushforward $\widetilde f_* D$ is $\FF$-Cartier (resp.~Cartier).
\item[(ii)] $\widetilde\mu^*(\widetilde f_* D) = \widehat f_*(\widetilde\rho^* D)$ holds.
\end{enumerate}
\end{lem}
\begin{proof}
If $D$ is an $\FF$-Cartier $\FF$-divisor then we can write $D=\sum d_i D_i$ with $d_i\in \FF$ and  $D_i$ Cartier. Since the assertions are $\FF$-linear in $D_i$, by replacing $D$ with $D_i$, we can assume that $D$ is Cartier. Denote 
\[
D_{\widetilde Z} = \widetilde f_* D \text{ and }D_{\widehat Z} =  \widehat f_*(\widetilde\rho^* D)
\]
Note that, $\widetilde f$ is a ramified cover and $D_{\widetilde Z}$ is nothing but the norm of $D$ under $\widetilde f$. The assertion  (i) (for $D$ Cartier) is already contained in \cite[2.40]{KollarMMP13}; see also \cite[\href{https://stacks.math.columbia.edu/tag/0BCX}{Tag 0BCX}]{stacks-project}. Let us recall the argument: one can find an open cover $\widetilde Z = \cup_i V_i$ such that $D\cap \widetilde f^{-1}(V_i)$ is a principal divisor of $\widetilde f^{-1}(V_i)$ defined by a local equation $g_i=0$, where $g_i\in \CC(\widetilde f^{-1}(V_i)) = \CC(\widetilde X)$ is a rational function. Then $D_{\widetilde Z}\cap V_i$ is a principal divisor of $V_i$ defined by $\Norm_{\widetilde f}(g_i)=0$, where $\Norm_{\widetilde f}(g_i)\in \CC(V_i) = \CC(\widetilde Z)$ is the norm of $g_i$ under $\widetilde f$. 

By Lemma~\ref{lem: factorize}, there is a commutative diagram of projective surjective morphisms between normal varieties
\begin{equation}\label{eq: sandwich}
\begin{tikzcd}
X' \arrow[r, "\widehat\rho"] \arrow[d, "f'"]& \widehat X\arrow[d, "\widehat f"] \arrow[r, "\widetilde \rho"] & \widetilde X\arrow[d, "\widetilde f"]\\
Z' \arrow[r, "\widehat\mu"] & \widehat Z \arrow[r, "\widetilde \mu"] & \widetilde Z
\end{tikzcd} 
\end{equation}
such that the morphisms $\widehat\rho$ and $\widehat\mu$ are birational and $f'$ is finite. Set
\[D':=(\widetilde\rho\circ \widehat\rho)^* D \text{ and }D_{Z'} = f'_* D'.
\]
On $(\widetilde f\circ \widetilde\rho\circ \widehat\rho)^{-1}(V_i)$, the divisor $D'$ is principal,  defined by $g_i'=0$, where $g_i'=(\widetilde\rho\circ \widehat\rho)^*g_i:=g_i\circ \widetilde\rho\circ \widehat\rho$ is the pull-back of $g_i$. Then, similarly as in the case of $D_{\widetilde Z}$ above, $D_{Z'} \cap (\widetilde\mu\circ\widehat \mu)^{-1}(V_i)$ is defined by $\Norm_{f'}(g_i')=0$. The identifications of the rational functions on $X'$ and $\widetilde X$ (resp.~$Z'$ and $\widetilde Z$) via pull-backs are compatible with the operation of taking norms, as the following commutative diagram shows:
\[
\begin{tikzcd}
\CC(\widetilde X) \arrow[r,"\cong","(\widetilde\rho\circ\widehat\rho)^*"']\arrow[d,"\Norm_{\widetilde f}"']&\CC(X')\arrow[d,"\Norm_{f'}"]\\
\CC(\widetilde Z) \arrow[r,"\cong","(\widetilde\mu\circ\widehat\mu)^*"']&\CC(Z') 
\end{tikzcd}
\]
Therefore,
\[
\Norm_{f'}(g_i') = \Norm_{f'}((\widetilde\rho\circ \widehat\rho)^*g_i) = (\widetilde\mu\circ\widehat\mu)^* \Norm_{\widetilde f}(g_i) 
\]
that is, the local defining equation of $D_{Z'}$ is the pullback of the local defining equation of $D_{\widetilde Z}$. In other words, we have $D_{Z'} = (\widetilde\mu\circ\widehat \mu)^*D_{\widetilde Z}$.  We do a simple verification that the two divisors $\widehat\mu_*D_{Z'}$ and $D_{\widehat Z}$ are the same:
\begin{equation*}
\widehat\mu_*D_{Z'}  = \widehat\mu_*f'_* D' = \widehat f_*\widehat\rho_*D' =\widehat f_*\widehat\rho_*(\widetilde\rho\circ \widehat\rho)^* D =  \widehat f_*\widetilde\rho^* D =  D_{\widehat Z}
\end{equation*}
where we use the commutativity of \eqref{eq: sandwich} for the second equality, and the fact that $\widehat\rho_*\widehat\rho^*$ acts as the identity on Cartier divisors for the fourth equality.
In conclusion, we obtain
\begin{equation*}
D_{\widehat Z} = \widehat\mu_* D_{Z'} = \widehat\mu_* (\widetilde\mu\circ\widehat \mu)^*D_{\widetilde Z} = \widetilde\mu^* D_{\widetilde Z},
\end{equation*}
which is what we wanted to prove.
\end{proof}

\begin{thm}\label{thm: gen finite}
Let $\FF$ be either the rational number field $\QQ$ or the real number field $\RR$. Let $S$ be a scheme. Let $f\colon X\rightarrow Z$ be a  generically finite surjective morphism between normal varieties, projective over $S$. Suppose that there are an $\FF$-divisor $B$ and a $\rmb$-$\FF$-divisor $\bfM$ on $X$, such that 
\begin{itemize}[leftmargin=*]
\item $\bfM$ is an $\FF_{> 0}$-linear combination of nef$/S$ $\QQ$-Cartier $\rmb$-divisors, and
\item $(X/S, B+\bfM)$ is a generalized pair such that $K_X+B+\bfM_X\sim_{\FF, f} 0 .$
\end{itemize}
Then there are an $\FF$-divisor $B_Z$ and a nef/$S$ $\rmb$-$\FF$-divisor $\bfM_f$ on $Z$ such that the following holds.
\begin{enumerate}[leftmargin=*]
\item[(i)] $\bfM_f$ is an $\FF_{>0}$-linear combination of nef$/S$ $\QQ$-Cartier $\rmb$-divisors.
\item[(ii)] $B_Z$ is effective if $B$ is.
\item[(iii)] $(Z/S, B_Z+\bfM_{f})$ is a generalized pair with $$K_X+B+\bfM_X\sim_{\FF} f^*(K_Z+B_Z+\bfM_{f,Z}).$$
\item[(iv)] $(Z/S, B_Z+\bfM_{f})$ is sub-lc (resp.~sub-klt) if $(X/S, B+\bfM)$ is.
\item[(v)] If $\bfM$ is semiample over $S$, then so is $\bfM_f$.
\end{enumerate}
\end{thm}
\begin{proof}
By Lemma~\ref{lem: R to Q}, we only need to consider the case $\FF=\QQ$.

Let $Z^\circ$ be the smooth locus of $Z$ and $X^\circ = f^{-1}(Z^\circ)$ its preimage.  Let $f|_{X^\circ}\colon X^\circ \rightarrow Z^\circ$ be the restriction of $f$. Then
  we have the Hurwitz formula 
\begin{equation}\label{eq: Hurwitz1}
K_{X^\circ} = (f|_{X^\circ})^*K_{Z^\circ} + R^\circ,
\end{equation}
where $R^\circ$ is the (effective) ramification divisor of $f|_{X^\circ}$. Let $R$ be the closure of $R^\circ$ in $X$. Now we define
\begin{equation}\label{eq: BZ}
B_Z  = \frac{1}{\deg f}  f_*(R + B).
\end{equation}
One sees immediately that $B_Z$ is effective if $B$ is. The $\rmb$-$\QQ$-divisor $\bfM_f$ is defined by specifying its traces on the higher birational models of $Z$ as follows: for any proper birational morphism $\mu\colon Z'\rightarrow Z$, consider the following commutative diagram
\[
\begin{tikzcd}
X'  \arrow[r, "\rho"] \arrow[d, "f'"]  & X \arrow[d, "f"]\\
 Z' \arrow[r, "\mu"]  & Z
\end{tikzcd}
\]
where $X'$ is the normalization of the main component of $X\times_Z Z'$. Define
\begin{equation}\label{eq: M}
\bfM_{f,  Z'} =\frac{1}{\deg f} f'_*\bfM_{X'},
\end{equation}
By the fact that $\bfM$ is a $\rmb$-$\QQ$-divisor, one readily verifies that the traces $\{\bfM_{f, Z'}\}_{Z'\rightarrow Z}$ are compatible with pushforward, so $\bfM_f$ is indeed a $\rmb$-$\QQ$-divisor. By \eqref{eq: Hurwitz1} we have $(f|_{X^\circ})_*K_{X^\circ} =(\deg f) K_{Z^\circ} + (f|_{X^\circ})_* R^\circ.$ Since $Z-Z^\circ$ has codimension at least two in $Z$,  
\begin{equation}\label{eq: Hurwitz2}
f_*K_{X} =(\deg f) K_{Z} + f_* R
\end{equation} 
also holds. 
Combining \eqref{eq: BZ}, \eqref{eq: M} and \eqref{eq: Hurwitz2}, we obtain
\begin{equation}\label{eq: Cartier}
  K_Z+B_Z+\bfM_{f,  Z} = \frac{1}{\deg f} f_*(K_{X} + B + \bfM_X).
\end{equation}
By the assumption that $K_X+B+\bfM_X\sim_{\QQ, f} 0$, there is a $\QQ$-Cartier $\QQ$-divisor $D$ on $Z$ such that $K_{X} + B + \bfM_X\sim_{\QQ} f^* D$. The projection formula together with \eqref{eq: Cartier} implies that $ K_Z+B_Z+\bfM_{f,  Z} \sim_{\QQ} D$ and thus 
\[
K_{X} + B + \bfM_X\sim_{\QQ} f^*(K_Z+B_Z+\bfM_{f,  Z}).
\]

We need to show that $\bfM_{f}$ is nef/$S$. Let $\rho\colon \widetilde X\rightarrow X$ be a higher birational model, to which $\bfM$ descends. By Lemma~\ref{lem: factorize}, up to replacing $\widetilde X$ by an even higher birational model, we can construct  the following commutative diagram
\begin{equation}\label{eq: descend}
\begin{tikzcd}
\widetilde X  \arrow[r, "\rho"] \arrow[d, "\widetilde{f}"]  & X \arrow[d, "f"]\\
\widetilde Z \arrow[r, "\mu"]  & Z
\end{tikzcd}
\end{equation}
such that $\mu$ is birational and $\widetilde f$ is finite. By Remark~\ref{rmk: fibre prod}, $\widetilde X$ is indeed the normalization of the main component of $X\times_Z \widetilde Z$ and $\widetilde f$ is induced from $f$. Since $\bfM_{\widetilde X}$ is nef$/S$ and $\widetilde f$ is a finite morphism, $\bfM_{f, \widetilde Z} = (1/\deg f) \widetilde f_*\bfM_{\widetilde X} $ is also nef/$S$.

Now we show that $\bfM_f$ descends to $\widetilde Z$. Let $\widehat\mu\colon \widehat Z \rightarrow \widetilde Z$ be a proper birational morphism. We consider the induced morphism $\widehat f\colon \widehat X\rightarrow \widehat Z$ fitting in the following commutative diagram:
\begin{equation}
\begin{tikzcd}
\widehat X  \arrow[r, "\widetilde \rho"] \arrow[d, "\widehat{f}"]  &\widetilde X  \arrow[r, "\rho"] \arrow[d, "\widetilde{f}"]  & X \arrow[d, "f"]\\
\widehat Z  \arrow[r, "\widetilde \mu"]   &\widetilde Z \arrow[r, "\mu"]  & Z
\end{tikzcd}
\end{equation}
By Lemma~\ref{lem: push&pull}, one has
\begin{equation}\label{eq: descend Z}
\bfM_{f, \widehat Z} = \widetilde \mu^*\bfM_{f, \widetilde Z},
\end{equation}
In other words, $\bfM_f$ descends to $\widetilde Z$.

Up to now we have shown that $(Z/S, B_Z + \bfM_f)$ is a generalized pair. To check that it has the right type of singularities, we look at the discrepancy of an arbitrary prime divisor $E$ over $Z$. First assume that $E$ is exceptional over $Z$. By replacing $\widetilde Z$ in the diagram \eqref{eq: descend} with a higher birational model we can assume that $E$ is a divisor on $\widetilde Z$. Let $F\subset \widetilde f^{-1}(E)$ be a prime divisor. Near the generic point of $F$, we compute as in the proof of \cite[5.20]{KM98}:
\begin{align*}
 K_{\widetilde X} &= \rho^*(K_X+B+\bfM_X) + a_F(X/S, B+\bfM) F \\
& \sim_{\QQ} \rho^*f^*(K_Z+B_Z+\bfM_{f, Z}) +a_F(X/S, B+\bfM) F, \\
 K_{\widetilde X} &=\widetilde f^*K_{\widetilde Z} + (r-1) F\\
  & = \widetilde f^*\mu^*(K_Z+B_Z+\bfM_{f, Z}) + ra_E(Z/S, B_Z+\bfM_f) F +(r-1) F\\
  &= \rho^*f^*(K_Z+B_Z+\bfM_{f, Z}) + (ra_E(Z/S, B_Z+\bfM_f) +(r-1))F.
\end{align*}
where is $r\leq \deg f$ is the ramification index of $\widetilde f$ along $F$. Now consider the Stein factorization $\widetilde X\xrightarrow{\rho'} X'\xrightarrow{f'} Z$, where $\rho'$ is a birational morphism between normal varieties with connected fibres and $f'$ is finite. Since $E$ is exceptional over $Z$, so is $F$ over $X'$. By the Negativity Lemma we can compare the coefficients of $F$ above and obtain
\[
 a_F(X/S, B+\bfM)+1 = r(a_E(Z/S, B_Z+\bfM_f)  +1).
\]
It follows that $ a_E(Z/S, B_Z+\bfM_f) \geq -1$ (resp.~$>-1$) if and only if the same inequality holds for  $ a_F(X/S, B+\bfM)$. 

Now we assume that $E$ is a component of $B_Z$. Let $B_i\subset f^{-1}(E)$ be the prime divisors of $X$ lying over $E$ and $d_i$ the degree of the generically finite morphism $f|_{B_i}\colon B_i\rightarrow E$. Then by \eqref{eq: BZ}
\[
a_E(Z/S, B_Z+\bfM_f) = -\mult_E B_Z = -\frac{\sum_id_i \mult_{B_i} B }{\deg f}.
\]
Since $\sum_i d_i\leq \deg f$, we infer that $a_E(Z/S, B_Z+\bfM_f) \geq -1$ (resp.~$>-1$) if  $\mult_{B_i} B\leq 1$ (resp.~$<1$) for each $i$.

In conclusion, $(Z/S, B_Z +\bfM_f)$ is sub-lc (resp.~sub-klt) if $(X/S, B+\bfM)$ is. 

If $\bfM$ is semiample, the divisor $\bfM_{\widetilde X}$ is semiample. Because  $\widetilde f$ is a finite morphism, $\bfM_{f, \widetilde Z} = (1/\deg f) \widetilde f_*\bfM_{\widetilde X}$ is also semiample by Lemma~\ref{lem: sa} below.
\end{proof}
\begin{rmk}
The construction of $B_Z$ in the proof of Theorem~\ref{thm: gen finite} is not the same as the one for morphisms with connected fibres (\cite[Section~4]{Filipazzi18gcbf}).
\end{rmk}
For the lack of reference, we give a proof of the the following lemma.
\begin{lem}\label{lem: sa}
Let $f\colon X\rightarrow Y$ be a finite morphism between proper normal varieties. Then, for a semiample $\RR$-divisor $D$ on $X$, the push-forward $f_*D$ is also semiample.
\end{lem}
\begin{proof}
Since $D$ is a semiample $\RR$-divisor, we can write $D=\sum_i a_i D_i$, where $a_i\in\RR_{>0}$ and each $D_i$ is a Cartier divisor such that $|D_i|$ is base point free. It suffices to prove that $f_*D_i$ is semiample for each $i$. In fact, we will show that $|f_*D_i|$ is base point free: let $p\in Y$ be an arbitrary point. Since $f$ is a finite morphism, $f^{-1}(p)$ is a finite set of points. Because $|D_i|$ is base point free, one can find an element $D_i'\in |D_i|$ such that $f^{-1}(p)\cap D_i'=\emptyset$. Then $\overline D_i:=f_*D_i'\in|f_*D_i|$ is a divisor such that $p\not\in \Supp(\overline D_i$).  It follows that $p$ is not a base point of $|f_*D_i|$.
\end{proof}

\begin{proof}[Proof of Theorem~\ref{thm: main}]
By Lemma~\ref{lem: R to Q}, we only need to consider the case $\FF=\QQ$. Let $f\colon X\xrightarrow{g}Y \xrightarrow{h} Z$ be the Stein factorization. By \cite{Filipazzi18gcbf} one can define a generalized pair $(Y/S, B_Y +\bfM_g)$ such that $K_X+ B+\bfM_X\sim_{\QQ} g^*(K_Y + B_{Y} + \bfM_{g,Y})$. Since $K_X+B+\bfM_X\sim_{\QQ, f} 0 $, it necessarily holds $K_Y + B_{Y} + \bfM_{g,Y} \sim_{\QQ, h} 0$. By Theorem~\ref{thm: gen finite} we obtain the required generalized pair $(Z/S, B_Z + \bfM_f):=(Z/S, B_Z + \bfM_h)$. Note that, in both steps of the construction, the type of singularities is inherited by the new generalized pairs.
\end{proof}

\section{Applications}\label{sec: subadj} 
\subsection{Subadjunction} A subadjunction formula for log canonical generalized pairs is proved in \cite[Theorems~1.5 and 6.7]{Filipazzi18gcbf} when the lc center is exceptional or when the underlying variety is $\QQ$-factorial klt and the lc center is projective. It is based on the fact that the morphism from the unique divisor with discrepancy $-1$ over an exceptional lc center has connected fibres. Using Theorem~\ref{thm: main} we obtain a subadjunction formula with some of the assumptions in \cite[Theorems~1.5 and 6.7]{Filipazzi18gcbf} removed:
\begin{thm}\label{thm: subadj}
Let $\FF$ be either the rational number field $\QQ$ or the real number field $\RR$. Let $(X/S, B+\bfM)$ be a log canonical generalized pair over a quasi-projective scheme $S$ such that $\bfM$ is an $\FF_{>0}$-linear combination of nef/$S$ $\QQ$-Cartier $\rmb$-divisors.  If $W$ is a log canonical center of $(X/S, B+\bfM)$ and $W^\nu$ its normalization, then there exists an effective $\RR$-divisor $B_{W^\nu}$ and an $\FF_{>0}$-linear combination $\bfM_\iota$ of nef/$S$ $\QQ$-Cartier $\rmb$-divisors on $W^\nu$ such that
\begin{enumerate}[leftmargin=*]
\item $(W^\nu/S, B_{W^\nu} + \bfM_\iota)$ is a lc generalized pair, where $\iota\colon W^\nu\rightarrow W\hookrightarrow X$ is the composite, and
\item $K_{W^\nu}+B_{W^\nu} + \bfM_{\iota, W^\nu}\sim_\FF \iota^*(K_X+B+\bfM_X)$ holds.
\end{enumerate}   
Moreover, if $W$ is a minimal lc center, then $(W^\nu/S, B_{W^\nu} + \bfM_\iota)$ is klt.
\end{thm}
\begin{proof}
By \cite[Propostion~1.27]{HM18weak} or \cite[Proposition 3.9]{HL18wzd} there is a $\QQ$-factorial dlt modification $\rho\colon \widetilde X\rightarrow X$ of $(X/S, B+\bfM)$, so that $(\widetilde X,\widetilde B)$ is a $\QQ$-factorial dlt pair and every $\rho$-exceptional divisor has coefficient $1$ in $\widetilde B$, where 
$\widetilde B$ is given by $K_{\widetilde X} +\widetilde B+\bfM_{\widetilde X} = \rho^*(K_X+B+\bfM_X)$. 

Let $V$ be an lc center of $({\widetilde X}/S, B+\bfM)$ that is minimal with respect to inclusion under the condition $\rho(V) = W$. Then $V$ is normal, and there is a morphism $f\colon V\rightarrow W^\nu$ factoring $\rho|_V\colon V\rightarrow W\subset X$:
\begin{equation}
\begin{tikzcd}
V  \arrow[rr,hook, ] \arrow[d, "f"']  & & {\widetilde X} \arrow[d, "\rho"]\\
W^\nu  \arrow[r]  \arrow[bend right]{rr}{\iota} &W \arrow[r, hook]  & X
\end{tikzcd}
\end{equation}
where $\iota\colon W^\nu\rightarrow W\hookrightarrow X$ is the composite.
 
By \cite[Lemma~3.2]{HM18weak}, there is a dlt generalized pair $(V/S, B_V + \bfM_{\widetilde\iota})$ such that 
\[
(K_{\widetilde X} +\widetilde B + \bfM_{{\widetilde X}})|_V \sim_\FF K_V+B_{V} +\bfM_{\widetilde\iota,V}. 
\]
By construction, we have $K_V+B_{V} +\bfM_{\widetilde\iota,V}\sim_\FF\rho^{*}(K_X +B + \bfM_X)|_{V}$ and hence $K_V+B_{V} +\bfM_{\widetilde\iota,V}\sim_{\FF, f} 0$. Now Theorem~\ref{thm: main} implies that there is a lc generalized pair $(W^\nu/S, B_{W^\nu} + \bfM_\iota)$ such that $\bfM_\iota$  is an $\FF_{>0}$-linear combination of nef/$S$ $\QQ$-Cartier $\rmb$-divisors on $W^\nu$, and
$$K_V+B_{V} +\bfM_{\widetilde\iota, V} \sim_\FF f^*(K_{W^\nu}+B_{W^\nu} + \bfM_{\iota, W^\nu}).$$ By construction, we necessarily have $$K_{W^\nu}+B_{W^\nu} + \bfM_{\iota, W^\nu}\sim_\FF \iota^*(K_X+B+\bfM_X).$$ 
If $W$ is a minimal lc center, then $(V/S, B_V +\bfM_{\widetilde\iota})$ is klt and it follows that $(W^\nu/S, B_{W^\nu} + \bfM_\iota)$ is also klt by Theorem~\ref{thm: main}.
\end{proof}

\subsection{Images of anti-nef lc generalized pairs}\label{sec: anti-ps}
As another application of Theorem~\ref{thm: main}, we show that the image of an anti-nef lc generalized pair has the structure of a numerically trivial lc generalized pair.  This implies \cite[Main Theorem]{CZ2013QofDPSJEMS} in the setting of lc generalized pairs (with $\RR$-coefficients); see also \cite[Corollary~5.3]{Filipazzi18gcbf}.

\begin{thm}\label{thm: CZ}
Let $(X/S,B+\bfM)$ be a log canonical generalized pair (resp.~klt generalized pair) over a quasi-projective scheme $S$ such that $\bfM$ is an $\RR_{>0}$-linear combination of $\QQ$-Cartier $\rmb$-divisors that are nef over $S$.  If $-(K_X+B+\bfM_X)$ is nef/$S$ and there is a projective surjective morphism $f\colon X\to Z$ onto a normal variety $Z$ over $S$, then there is a log canonical generalized pair (resp.~klt generalized pair) $(Z/S,B_{Z}+\bfM_{f})$ with $\QQ$-coefficients such that $K_Z+B_{Z}+\bfM_{f,Z} \sim_{\QQ, S} 0$. In particular, if $Z$ is $\Qq$-Gorestein then $-K_Z$ is pseudo-effective over $S$.
\end{thm}
\begin{proof}
	Replacing $(X/S,B+\bfM)$ by its $\QQ$-factorial dlt modification, we may assume that $X$ is $\Qq$-factorial klt. 
	
		By Corollary~\ref{cor: rat decomp}, there exist finitely many real numbers $c_k\in\RR_{>0}$ and nef$/S$ Cartier divisors $N^{(k)}$ on $X$ such that $$-(K_X+B+\bfM_X)=\sum_k c_k N^{(k)}.$$ Let $\bfN^{(k)}= \overline {N^{(k)}}$ be the Cartier closure of $N^{(k)}$ and $\bfN=\bfM+\sum_k c_k \bfN^{(k)}$.  Then $(X/S,B+\bfN)$ is lc and $\bfN$ is an $\RR_{>0}$-linear combination of  $\QQ$-Cartier $\rmb$-divisors that are nef over $S$. By the construction of $\bfN$, we have $$K_X+B+\bfN_X=K_X+B+\bfM_X+\sum_k c_k N^{(k)}=0.$$ By Lemma~\ref{lem: R to Q} there is a lc generalized pair $(X/S,  B'+ \bfN')$  with $\QQ$-coefficients such that $K_X+ B' + \bfN'_X\sim_{\QQ, S} 0 $. 
			
By Theorem~\ref{thm: main}, there is a lc generalized pair $(Z/S,B_{Z}+\bfM_{f})$ with $\QQ$-coefficients such that
	$$0\sim_{\QQ, S} K_X+B' +\bfN'_X\sim_{\QQ} f^{*}(K_Z+B_{Z} +\bfM_{f,Z}).$$
It follows that $K_Z+B_{Z} +\bfM_{f,Z} \sim_{\QQ, S} 0$. If $Z$ is $\QQ$-Gorenstein then $-K_Z\sim_{\QQ, S} B_{Z} +\bfM_{f,Z}$ is pseudo-effective over $S$.

Finally, if $(X/S, B+\bfM)$ is klt, then $(X/S,B+\bfN)$ is klt. Consequently, the generalized pairs $(X/S, B'+\bfN')$ and  $(Z/S,B_{Z}+\bfM_{f})$ can both be made klt by Lemma~\ref{lem: R to Q} and Theorem~\ref{thm: main}.
\end{proof}

\bibliographystyle{cdraifplain}
	\bibliography{bibfile}
\end{document}